\documentclass{amsart}

\usepackage{amssymb} 
\usepackage[mathscr]{euscript}
\usepackage{calligra}

\theoremstyle{plain}
\newtheorem{theorem}{Theorem}[section]
\newtheorem{lemma}[theorem]{Lemma}
\newtheorem{corollary}[theorem]{Corollary}
\newtheorem{proposition}[theorem]{Proposition}

\theoremstyle{definition}
\newtheorem{definition}[theorem]{Definition}

\newtheorem{problem}[theorem]{Problem}

\theoremstyle{remark}

\newcommand{\M}{\text{\calligra{Micro}}} 
\newcommand{\Mf}{\text{\calligra{Micro}}_{(f_n)_n}}
\newcommand{\Nano}{\text{\calligra{Nano}}} 
\newcommand{\Pico}{\text{\calligra{Pico}}}

\renewcommand{\subset}{\subseteq}
\renewcommand{\supset}{\supseteq}

\begin{document}

\title{Ideal-like properties of generalized microscopic sets}

\author[K. Czudek]{Klaudiusz Czudek}
\address{Institute of Mathematics, University of Gda\'{n}sk, ul.~Wita  Stwosza 57, 80-952 Gda\'{n}sk, Poland}
\email{klaudiusz.czudek@gmail.com}

\author[A. Kwela]{Adam Kwela}
\address{Institute of Mathematics, University of Gda\'{n}sk, ul.~Wita Stwosza 57, 80-952 Gda\'{n}sk, Poland}
\email{Adam.Kwela@ug.edu.pl}

\author[N. Mro\.{z}ek]{Nikodem Mro\.{z}ek}
\address{Institute of Mathematics, University of Gda\'{n}sk, ul.~Wita Stwosza 57, 80-952 Gda\'{n}sk, Poland}
\email{nmrozek@mat.ug.edu.pl}

\author[W. Wo{\l}oszyn]{Wojciech Wo{\l}oszyn}
\address{Institute of Mathematics, University of Gda\'{n}sk, ul.~Wita Stwosza 57, 80-952 Gda\'{n}sk, Poland}
\email{woloszyn.w.a@gmail.com}

\date{\today}

\thanks{The third author was supported by grant BW-538-5100-B860-15}


\begin{abstract}
We show that not every family of generalized microscopic sets forms an ideal. Moreover, we prove that
some of these families have some weaker additivity properties and some of them do not have even that.
\end{abstract}

\maketitle

\section{Introduction}

By $\omega$ we denote the set of natural numbers, i.e., $\omega=\{0,1,\ldots\}$. For an interval $J\subset\mathbb{R}$ by $|J|$ we denote its length.

\begin{definition}
A set $M\subseteq\mathbb{R}$ is called \emph{microscopic} if for every $\varepsilon>0$ there is a sequence of intervals $(I_k)_k$ such that $M\subset\bigcup_k I_k$ and $|I_k|\leq \varepsilon^k$ for all $k\in\omega$. The family of all microscopic sets will be denoted by $\M$.
\end{definition}

The above notion was introduced in 2001 by J.~Appell (cf.~\cite{firstmicro}). In~\cite{mikro2} J.~Appell, E.~D'Aniello and M.~V\"ath studied connections between microscopic sets (as well as several other notions of small sets) and some kinds of continuity of real functions. Many properties of microscopic sets are similar to those of classical Lebesgue null sets. In particular, the family of microscopic sets form a $\sigma$-ideal that lies between $\sigma$-ideals of strong measure zero sets and null sets. More similarities were found in~\cite{duality}. There are also some differences. For instance, recently, one of the authors of this paper (A. Kwela) proved that additivity of $\M$ is $\omega_1$ (cf. \cite{add}). A good survey on microscopic sets can be found in ~\cite[Chapter 20]{monografia}.

In 2014 G. Horbaczewska introduced the following nice generalization of microscopic sets (cf.~\cite{mikrohorba}).
Let $(f_n)_n$ be a nonincreasing sequence of functions $f_n\colon (0,1)\to (0,1)$ such that:
\begin{itemize}
 \item $f_n$'s are increasing;
 \item $\lim_{x\to 0^{+}} f_n(x)=0$ for all $n$;
 \item there exists $x_0$ such that for every $x\in (0,x_0)$ the series $\sum_{n\in\omega} f_n(x)$ is convergent.
\end{itemize}
All sequences of functions considered in this paper are supposed to satisfy such conditions.

\begin{definition}
A set $M\subseteq\mathbb{R}$ is called \emph{$(f_n)_n$-microscopic} if for every $\varepsilon\in (0,1)$ there is a sequence of intervals $(I_k)_k$ such that $M\subset\bigcup_k I_k$ and $|I_k|\leq f_k(\varepsilon)$ for all $k\in\omega$. The family of all $(f_n)_n$-microscopic sets will be denoted by $\Mf$.
\end{definition}

In~\cite{mikrohorba} Horbaczewska studied some basic properties of such families of sets. In particular, she gave some conditions on the sequence $(f_n)_n$, under which the family $\Mf$ is equal to $\M$, and proved that this notion generates new families of sets. For many of such families she showed that they form $\sigma$-ideals, but for example for the sequence $(x^{2^n})_n$ her proof does not work. She asked whether such a family forms a $\sigma$-ideal. That was a starting point for our research. We investigate problems of the following sort: given a sequence $(f_n)_n$ and two sets, $A\in\Mf$ and $B$ small in some sense, is it true that $A\cup B\in\Mf$?

The paper is organized as follows. In Chapter \ref{general} we show some basic properties of generalized microscopic sets and prove that for any sequence $(f_n)_n$ the family of all sets that can be covered by an $\mathbf{F_\sigma}$ set from $\Mf$ forms a $\sigma$-ideal.

In Chapter~\ref{nano} we answer the question of Horbaczewska. For simplification, we call an $(x^{2^n})_n$-microscopic set \emph{nanoscopic} and denote the family of all nanoscopic sets by $\Nano$. We show that $\Nano$ does not even form an ideal. Of course, the family of nanoscopic sets is closed under taking subsets, but we will construct two nanoscopic sets union of which is not nanoscopic anymore. In the same chapter we show that some additivity properties remain true for $\Nano$. In particular, we show that union of a nanoscopic set and a strong measure zero set remains a nanoscopic set. Also, we give some conditions imposed on an $(f_n)_n$-microscopic set under which its union with any set of strong measure zero remains $(f_n)_n$-microscopic.

In the last chapter we show that some families of generalized microscopic sets are so far from being an ideal, that even adding a point to such a set can cause that it is not in this family anymore. In particular, we focus on a family of sets generated by the sequence $(x^{n!})_n$ and call its members \emph{picoscopic sets}. This family will be denoted by $\Pico$.


\section{General properties}
\label{general}


We start with some properties that are true for all families of generalized microscopic sets and do not depend on particular sequence of functions. 

The first fact has been already observed by Horbaczewska, but is unpublished, so we prove it here.

\begin{proposition}
For any sequence $(f_n)_n$ the family $\Mf$ is closed under taking subsets and $\mathbf{G_\delta}$-generated. 
\end{proposition}

\begin{proof}
Fix a sequence $(f_n)_n$. The first part is obvious.

We will show that the family of all $\Mf$ sets is $\mathbf{G_\delta}$-generated. Indeed, if $A\in \Mf$, then for each $n\in\omega$ let $(I^n_k)_k$ be a sequence of open intervals such that $A\subset\bigcup_k I^n_k$ and $|I^n_k|\leq f_k(\frac{1}{n+1})$ for all $k\in\omega$. The set $G=\bigcap_n\bigcup_k I^n_k$ is a $\mathbf{G_\delta}$ set and $A\subset G$. We will show that $G\in\Mf$. If $\varepsilon>0$, then there is $n$ with $\varepsilon>\frac{1}{n+1}$. It suffices to observe that the sequence of intervals $(I^n_k)_k$ covers the set $G$ and $|I^n_k|\leq f_k(\frac{1}{n+1})\leq f_k(\varepsilon)$.
\end{proof}

Now we present a result which shows relationship between generalized microscopic sets and the family of Lebesgue null sets. 

\begin{proposition}
\label{null}
Every $(f_n)$-microscopic set is of Lebesgue measure zero. Moreover, if $X$ is a set of Lebesgue measure zero, then there exists a sequence $(g_n)_n$ such that $X$ is $(g_n)_n$-microscopic.
\end{proposition}

\begin{proof}
At first, take any $(f_n)$-microscopic set $X$. Notice that we can assume that functions $f_n$ are defined on the interval $[0,1)$ and $f_n(0)=0$ for all $n\in\omega$. From the definition of $(f_n)$-microscopic set, there is $x_0\in (0,1)$ such that the series $\sum_{n\in\omega}f_n(x_0)$ converges and $f_n(x)\leq f_n(x_0)$ for all $n\in\omega$ and $x\in [0,x_0]$. Therefore, $\sum_{n\in\omega}f_n$ converges uniformly on $[0,x_0]$. We conclude that $\sum_{n\in\omega}f_n$ is continuous in 0 and, finally, that $X$ must be of Lebesgue measure zero.

To prove the second part, let $(I_n^m)_n$, for $m\in\omega$, be a sequence of nonempty open intervals covering $X$ such that $\sum_{n\in\omega}|I_n^m|<\frac{1}{2^{m+1}}$ and $|I_{n+1}^m|\leq|I^m_n|$ for all $n, m\in\omega$. Define $a_{m, n}=\max\{|I^j_n| : j\geq m\}$ for each $n,m\in\omega$. It is easy to see that $a_{m, n}$'s have the following properties:
\begin{itemize}
	\item $\sum_{j\in\omega} a_{m,j}<\infty$ for each $m\in\omega$;
	\item $\lim_{j\to\infty}a_{j,n}=0$ for each $n\in\omega$;
	\item $a_{m+1, n}\leq a_{m, n}$ for all $n, m\in\omega$;
	\item $|I_n^m|\leq a_{m, n}$ for all $n, m\in\omega$.
\end{itemize}

Take any $n, m\in\omega$. Let $k\geq m$ be the maximal natural number such that $a_{m, n+1}=a_{k, n+1}$. Then we have:
$$a_{m, n+1}=a_{k, n+1}=|I_{n+1}^k|\leq|I_n^k|\leq \max\{|I_n^j| : j\geq k\}\leq \max\{|I_n^j| : j\geq m\}=a_{m, n}.$$
Now let $(g_n)_n$ be any sequence of functions satisfying conditions from the definition of $(f_n)$-microscopic sets and such that $g_n(\frac{1}{2^{m+2}})=a_{m, n}$ for $m, n\in\omega$. Properties mentioned above guarantee that this construction is possible and $X$ is $(g_n)_n$-microscopic.
\end{proof}

If $\mathcal{A}$ is a family of subsets of reals, then by $\mathcal{A}^{\star}$ we mean a family of all sets that can be covered by an $\mathbf{F_\sigma}$ set from $\mathcal{A}$.
It is known that if $\mathcal{I}$ is a $\mathbf{G_\delta}$ generated $\sigma$-ideal that contains all singletons, then 
$\mathcal{I}^{\star}$ forms a $\sigma$-ideal as well (see~\cite{Fsigma-ideal}).
Our next goal is to show that for any sequence $(f_n)_{n}$ the family $\Mf^{\star}$ is a $\sigma$-ideal. In the case of microscopic sets, it is obvious since they form a $\sigma$-ideal.
In the case of nanoscopic sets, it was already observed by G. Horbaczewska in~\cite{mikrohorba}. However, her proof does not work in the general case, so we present here an essentially new one.
We start with an observation that even if both $\Mf$ and $\Mf^{\star}$ are $\sigma$-ideals, they always differ.

\begin{proposition}
For any sequence $(f_n)_n$ there exists a set $X$ that belongs to $\Mf$ and does not belong to $\Mf^{\star}$.
\end{proposition}

Our proof is rather standard.

\begin{proof}
Fix a sequence $(f_n)_n$ and an enumeration $\mathbb{Q}\cap [0,1]=\{q_i:\ i\in\omega\}$. Let $X=\bigcap_n \bigcup P_n$, where
$$P_n=\left\{\left(q_i-\frac{1}{2}f_{i}\left(\frac{1}{n+1}\right),q_i+\frac{1}{2}f_{i}\left(\frac{1}{n+1}\right)\right) :\ i\in\omega\right\}.$$
It is easy to see that the set $X$ is $(f_n)_n$-microscopic.

On the other hand, $X$ is cannot be included in any $(f_n)_n$-microscopic $\mathbf{F_\sigma}$ set. Indeed, assume otherwise and let $X\subset\bigcup_k F_k$, where each $F_k$ is closed and $\bigcup_k F_k$ is $(f_n)_n$-microscopic. Observe that each $F_k$ is nowhere dense (otherwise it would contain an open interval of positive Lebesgue measure and each $F_k$ is of measure zero as an $(f_n)_n$-microscopic set; cf. Proposition \ref{null}). Hence, $\bigcup_k F_k$ is of first category. However, $X$ is a $\mathbf{G_\delta}$ set which is dense in $[0,1]$, so it is residual in $[0,1]$. Therefore, $[0,1]\setminus X$ is of first category. We get that $[0,1]$ is a union of two sets of first category. This contradicts the Baire Theorem.
\end{proof}

Now we proceed to showing that the family $\Mf^{\star}$ always forms a $\sigma$-ideal. We will need the following technical lemma.

\begin{lemma}
\label{zwarte}
Let $X$ be a compact $(f_n)$-microscopic set. Then for every $k\in\omega$ and $\varepsilon\in (0,1)$ the set $X$ can be covered by some finite subsequence of a sequence of intervals of lengths $(f_n(\varepsilon))_{n> k}$.
\end{lemma}

\begin{proof}
Choose $\varepsilon>0$ and $k\in\omega$. We will show that $X$ can be covered by open intervals $(I_n)_{k<n}$ such that $|I_n|<f_n(\varepsilon)$. By compactness of $X$, it will end the proof. Firstly, for some technical reasons we need to introduce two covers of $X$.
 
Let $\varepsilon'$ be such that $f_0(\varepsilon')<f_{2k}(\varepsilon)$. Since $X$ is a compact $(f_n)$-microscopic set, we can find $l'\in\omega$ with $l'>k$ and open intervals $(I'_n)_{n\leq l'}$ such that $|I'_n|<f_n(\varepsilon')$, for each $n\leq l'$, that cover the set $X$.
 
Let $\varepsilon''$ be such that $f_0(\varepsilon'')<f_{k+l'}(\varepsilon)$. Now we can find $l''\in\omega$ and an open cover $(I''_n)_{n\leq l''}$ of $X$ such that $|I''_n|<f_n(\varepsilon'')$ for all $n\leq l''$. 
Without loss of generality, we can assume that $l''>k+l'$ and that every interval $I''_i$ is included in some $I'_j$ (to ensure the second part it suffices to take $\varepsilon''$ such that $f_0(\varepsilon'')$ is smaller than gaps between intervals $I'_j$, for $j\leq l'$).

By the pigeonhole principle, we can find numbers $n_0,\ldots,n_{k-1}\in \{0,\ldots,l'\}$
and $m_0\ldots,m_{2k-1}\in\{0,\ldots,k+l'\}$ such that each interval $I''_{m_i}$ is included in some $I'_{n_j}$.

We are ready to build a cover of $X$ by sets $(I_n)_{n>k}$ such that $|I_n|<f_n(\varepsilon)$ for all $n$.
Let $I_{k+1},\ldots,I_{2k}$ be intervals such that they have the required length and cover intervals $I'_{n_0},\ldots,I'_{n_{k-1}}$ (recall that $f_0(\varepsilon')<f_{2k}(\varepsilon)$).
Hence, they cover also intervals $I''_{m_0},\ldots,I''_{m_{2k-1}}$. Now we can find intervals $I_{2k+1},\ldots,I_{k+l'}$ of the required lengths that cover all the intervals from $(I''_{n})_{n<k+l'}$ which are not covered by $I_{k+1},\ldots,I_{2k}$ (note that there are at most $k+l'-2k$ such $I''_{n}$ and recall that $f_0(\varepsilon'')<f_{k+l'}(\varepsilon)$). Finally, let $I_{n}=I''_{n}$ for all $n\in\{k+l',\ldots,l''\}$. Hence, we have built the desired cover of the set $X$. 
\end{proof}

\begin{theorem}
\label{Fsigma}
For any sequence $(f_n)_n$ the family $\Mf^{\star}$ forms a $\sigma$-ideal.
\end{theorem}

\begin{proof}
Fix a sequence $(f_n)_n$. It is obvious that $\Mf^{\star}$ is closed under taking subsets. 

We have to show that $\Mf^{\star}$ is closed under countable unions. Let $A_i\in\Mf^{\star}$ for all $i\in\omega$ and choose $\varepsilon\in (0,1)$. For each $A_i$ we can find compact $(f_n)_n$-microscopic sets $B^i_k$ such that
$A_i\subset\bigcup_k B^i_k$. Let $(C_k)_{k\in\omega}=(B_k^i)_{k,i\in\omega}$ be a reenumeration.
By Lemma~\ref{zwarte}, we can find intervals $I_n$ such that $|I_n|<f_n(\varepsilon)$, for each $n$, and 
$$\bigcup A_n\subset\bigcup C_n\subset\bigcup I_n.$$
This ends the proof.
\end{proof}


\section{Nanoscopic sets}
\label{nano}


\begin{definition}
A set $M$ is called \emph{nanoscopic} if for any $\varepsilon>0$ there exists a sequence of intervals $(I_k)_k$ such that $M\subset\bigcup_k I_k$
and $|I_k|\leq \varepsilon^{2^k}$ for all $k\in\omega$. In this case we write $M\in\Nano$.
\end{definition}

A set $M$ is of strong measure zero if for each sequence of positive reals $(\varepsilon_k)_k$ there is a sequence of intervals $(I_k)_k$ such that $M\subset\bigcup_k I_k$ and $|I_{k}|\leq \varepsilon_k$ for all $k\in\omega$. Observe that all countable sets are of strong measure zero.

\begin{theorem}
\label{smz-nano}
If $A$ is nanoscopic and $B$ is of strong measure zero, then $A\cup B$ is nanoscopic.
\end{theorem}

\begin{proof}
For each $n\in\omega$ let $(I^n_k)_k$ be a sequence of intervals such that $A\subset\bigcup_k I^n_k$ and $|I^n_k|\leq (\frac{1}{2^{2^n}})^{2^k}=\frac{1}{2^{2^{n+k}}}$ for all $k\in\omega$. Note that $G=\bigcap_n\bigcup_k I^n_k$ is nanoscopic and $A\subset G$. Choose $\varepsilon>0$. There are two possible cases:

\textit{Case 1.} There is $n$ such that for all $k$ and $m>n$ only finitely many of the intervals $I^m_j$, for $j\in\omega$, are contained in the interval $I^n_k$. Then for each $k$ the set
$$A_k=I^n_k\cap G$$
is compact and nanoscopic (as a subset of $G$). Moreover, $A\subset\bigcup_k A_k$. 
By Lemma~\ref{zwarte} applied to $A_n$'s, we can find a sequence of naturals $(s_k)_k$ and intervals $I_n$, for all $n\in\omega\setminus\{s_0,s_1,\ldots\}$, such that:
\begin{itemize}
	\item $|I_k|<\varepsilon^{2^k}$ for each $k$;
	\item $A_n\subset \bigcup_{k\in\{s_n+1,\ldots,s_{n+1}-1\}} I_k$ for each $n$.
\end{itemize}
Now, by the definition of a strong measure zero set, we can find intervals $I_{s_n}$, for $n\in\omega$, that cover the set $B$ and satisfy $|I_{s_n}|<\varepsilon^{2^{s_n}}$. Hence, $A\cup B$ is nanoscopic.

\textit{Case 2.} For all $n$ there are $k$ and $m>n$ such that infinitely many of the intervals $I^m_j$, for $j\in\omega$, are contained in the interval $I^n_k$.

There is $n$ such that $\varepsilon>\frac{1}{2^{2^n}}$. Then, there are $k\in\omega$, $m>n$ and an infinite set $T=\{t_0,t_1,\ldots\}\subset\omega$ such that $I^n_k\supset I^m_j$ for each $j\in T$. Observe that $A\subset I^n_k\cup\bigcup_{j\in\omega\setminus T} I^m_j$. Let $\varepsilon_i=\frac{1}{2^{2^{m+t_i}}}$. 

Since $B$ is of strong measure zero, there is a sequence of intervals $(J_{i})_i$ such that $B\subset\bigcup_i J_{i}$ and $|J_i|\leq \varepsilon_i$ for all $i\in\omega$. 

Let $K_0=I^n_k$ and $K_{j+1}=I^m_j$ for $j\notin T$ and $K_{t_i+1}=J_{i}$ for $i\in\omega$. Then, the sequence of intervals $(K_i)_i$ is such that $A\cup B\subset\bigcup_i K_i$ and $|K_i|\leq \varepsilon^{2^i}$ for all $i\in\omega$ (note that $m\geq n+1$). Hence, $A\cup B$ is nanoscopic.
\end{proof}

Notice that in the first case of the above proof we do not use any specific properties of nanoscopic sets. Hence, we get the following general corollary.

\begin{corollary}
\label{cor}
If $A\in\Mf^{\star}$ and $B$ is of strong measure zero, then $A\cup B$ is $(f_n)_n$-microscopic.
\end{corollary}

In the next theorem we consider some other cases in which union of an $(f_n)_n$-microscopic set and a set of strong measure zero is $(f_n)_n$-microscopic.

\begin{theorem}
\label{suma}
Let $X$ be an $(f_n)$-microscopic set satisfying at least one of the following conditions:
\begin{itemize}
	\item[(a)] $X$ is in $\Mf^{\star}$;
	\item[(b)] $\overline{X}$ is an unbounded interval;
	\item[(c)] $X$ is bounded.
\end{itemize}
Then for any set $Y$ of strong measure zero the union $X\cup Y$ is $(f_n)$-microscopic.
\end{theorem}

\begin{proof}
{\bf (a):} This is Corollary \ref{cor}.

In the proofs of (a) and (b) we use the following observation: if $\varepsilon\in (0,1)$, $X$ is an $(f_n)$-microscopic set and $Y$ is of strong measure zero, then to find a cover $(J_n)_n$ of $X\cup Y$ such that $|J_n|\leq f_n(\varepsilon)$, it suffices to find a cover $(J'_n)_n$ of $X$ such that $|J'_n|\leq f_n(\varepsilon)$ and $X\subseteq \bigcup_{n\not\in T} J'_n$ for some infinite set $T\subseteq\omega$.

{\bf (b):}  Let $\varepsilon\in (0,1)$. Set $\delta\in (0,1)$ such that $f_0(\delta)<\frac{1}{3}f_0(\varepsilon)$ and take any cover $(I_n)_n$ of $X$ satisfying $|I_n|<f_n(\delta)$ for all $n$. Then $|I_n|<f_0(\delta)<\frac{1}{3}f_0(\varepsilon)$ for every $n$. Since $X$ is dense in an unbounded interval, there is $m\in\omega$ and an interval $I$ of length $f_0(\varepsilon)$ such that $I_0\cup I_m\subseteq I$.

There is also an interval $J$ of length $f_m(\varepsilon)$ containing infinitely many $I_n$'s. Indeed, otherwise let $(J_n)_n$ be any sequence of nonoverlapping closed intervals of length $f_m(\varepsilon)$ such that $J_n\subseteq \overline{X}$ for every $n$. Then $\{k\in\omega :\ I_k\cap J_n\neq\emptyset\}$ is finite for each $n\in\omega$. Hence, since $X$ is dense in an unbounded interval $\overline{X}$, we get that $|J_n|\leq\sum_k |I_k\cap J_n|$ for every $n\in\omega$. But then 
$$\sum_{n} |J_n|\leq\sum_n\sum_k |I_k\cap J_n|=\sum_k\sum_n |I_k\cap J_n|\leq\sum_k |I_k|<\infty,$$ 
which contradicts the fact that $\sum_{n} |J_n|=\infty$.

We are ready to define the required cover of $X$. Let $I'_0=I$, $I'_m=J$ and $I'_n=I_n$ for all $n\neq 0,m$. Observe that $X\subseteq \bigcup_{n\not\in T} I'_n$ for $T=\{n\in\omega:\ I_n\subseteq J\}$, and hence, we are done.

{\bf (c):}  Let $\varepsilon>0$. There is $r>0$ such that $X\subseteq \overline{B}(0,r)$. For each $k\in\omega$ set a cover $(I^k_n)_n$ of the set $X$ such that $I^k_n\subseteq \overline{B}(0,r)$ and $|I^k_n|\leq f_n(\frac{\varepsilon}{2^k})$ for every $n, k\in\omega$. If for some $k$ only finitely many $I^k_n$'s are not empty, then we are done. Therefore, we can assume that these covers are infinite. Consider the following two cases. 

Assume first that there are an open interval $J$ and $N,K\in\omega$ such that $|J|\leq f_0(\varepsilon)$ and $I^K_0, I^K_N\subseteq J$. Then we can define a new cover $(I_n)_n$ of $X$ as follows. Let $I_0=J$, $I_n=I^K_n$ for all $n\neq 0,N$, and let $I_N$ be any interval of length $f_N(\varepsilon)$ containing infinitely many $I^K_n$'s. Such interval exists since $\bigcup_{n\in\omega}I^K_n\subseteq \overline{B}(0,r)$. Observe that $X\subseteq \bigcup_{n\not\in T} I_n$ for $T=\{n\in\omega:\ I^K_n\subseteq I_N\}$. Hence, we are done by the observation we made before the proof. 

Now consider the second case: there are no $N, K\in\omega$ and $J$ of length not greater than $f_0(\varepsilon)$ such that $I^K_0, I^K_N\subseteq J$. For every $k\in\omega$ pick any $x_k\in I^k_0$. The sequence $(x_k)_k$ is contained in $\overline{B}(0, r)$ so there are $x\in\mathbb{R}$ and an increasing sequence of natural numbers $(m_k)_k$ such that $x_{m_k}\to x$. We can additionally assume that $|I^{m_0}_0|\leq\frac{1}{4}f_0(\varepsilon)$ and $|x_{m_0}-x|\leq\frac{1}{4}f_0(\varepsilon)$. Notice that $(x-\frac{1}{2}f_0(\varepsilon),x+\frac{1}{2}f_0(\varepsilon))$ contains at most one element of $X$. Hence, $X\cap (x-\frac{1}{2}f_0(\varepsilon),x+\frac{1}{2}f_0(\varepsilon))$ can be covered by any interval. Thus, we can define a new cover $(I_n)_n$ of $X$ as follows. Let $I_0$ be any interval of length $f_0(\varepsilon)$ containing infinitely many sets from $(I_n^{m_0})_n$. As before, such interval exists since $\bigcup_{n\in\omega}I^{m_0}_n\subseteq \overline{B}(0,r)$. Further, take any natural number $l$ such that $I^{m_0}_l\subseteq I_0$ and define $I_l$ as any interval of length $f_l(\varepsilon)$ containing $X\cap (x-\frac{1}{2}f_0(\varepsilon),x+\frac{1}{2}f_0(\varepsilon))$. For $n\neq 0,l$ define $I_n=I^{m_0}_n$. As in the previous case, $X\subseteq \bigcup_{n\not\in T} I_n$ for $T=\{n\in\omega:\ I^{m_0}_n\subseteq I_0\}$ and the entire proof is completed.
\end{proof}

Notice that using a similar argument to the one presented in the proof of part (b), we can show that if $Y$ is of strong measure zero, $X$ is $(f_n)$-microscopic and there is $\delta>0$ such that $\overline{X}$ is a union of a family of closed intervals of length greater than $\delta$, then $X\cup Y$ is $(f_n)$-microscopic. However, we do not know if this can be strengthened even further. 

\begin{problem}
Assume that $Y$ is of strong measure zero, $X$ is $(f_n)$-microscopic and there is an interval $I$ such that $I\subseteq\overline{X}$. Does $X\cup Y$ always belong to $\Mf$?
\end{problem}

The next theorem is an answer to a problem posed by Horbaczewska in~\cite{mikrohorba}.

\begin{theorem}
\label{tw}
The family of nanoscopic sets is not an ideal, i.e., there are two nanoscopic sets such that their union is not nanoscopic.
\end{theorem}

The proof is based on two lemmas.

\begin{definition}
\label{def}
Fix a sequence $(f_n)_n$ and $m\in\omega$. A set $M$ is called \emph{$m$-$(f_n)_n$-microscopic} if for every $\varepsilon\in (0,1)$ there is a sequence of intervals $(I_k)_k$ such that $M\subset\bigcup_k I_k$ and $|I_{mk}|,\ldots,|I_{mk+m-1}|\leq f_k(\varepsilon)$ for all $k\in\omega$. A $m$-$(x^{2^n})_n$-microscopic set is called \emph{$m$-nanoscopic}.
\end{definition}

\begin{lemma}
\label{lem1}
For any sequence $(f_n)_n$ if $X$ is a compact $m$-$(f_n)_n$-microscopic set, then there are $(f_n)_n$-microscopic sets $A_0,A_1,\ldots,A_{m-1}$ such that $X=\bigcup A_k$.
\end{lemma}

\begin{proof}
Fix a sequence $(f_n)_n$. We define inductively an increasing sequence of natural numbers $(l_k)_k$ and a sequence of closed intervals $(I_k)_k$.

Let $l_0=0$. Since $X$ is compact and $m$-$(f_n)_n$-microscopic, there is $l_1>l_0$ and a finite sequence of closed intervals $(I_k)_{ml_0\leq k<ml_1}$ such that $X\subset\bigcup_{ml_0\leq k<ml_1} I_k$ and 
$$|I_{mk}|,\ldots,|I_{mk+m-1}|=f_k\left(\frac{1}{2^{0+1}}\right)$$
for all $l_0\leq k<l_1$. Suppose that $l_i$ and intervals $I_k$, for $i\leq n$ and $k<ml_n$, are constructed. By compactness of $X$ and Lemma \ref{zwarte}, there is $l_{n+1}>l_n$ and a finite sequence of closed intervals $(I_k)_{ml_n\leq k<ml_{n+1}}$ such that $X\subset\bigcup_{ml_n\leq k<ml_{n+1}} I_k$ and 
$$|I_{mk}|,\ldots,|I_{mk+m-1}|=f_k\left(\frac{1}{2^{n+1}}\right)$$
for all $l_n\leq k<l_{n+1}$.

Define:
$$A_0=\bigcap_i\bigcup_{k\geq i}I_{mk}\cap X,$$
$$\vdots$$
$$A_{m-1}=\bigcap_i\bigcup_{k\geq i}I_{mk+m-1}\cap X.$$

Observe that $X=\bigcup A_k$. Indeed, if $x\in X$, then there is a sequence $(t_i)_{i}$ such that $x\in\bigcap_i I_{t_i}$. There exists $j<m$ such that infinitely many $t_i$'s are of the form $t_i=mk+j$ for some $k\in\omega$. Then, we get that $x\in A_j$.

We will show that each of the sets $A_j$ is $(f_n)_n$-microscopic. Fix $j<m$.
Given $\varepsilon\in (0,1)$ there is $n$ such that $\varepsilon>\frac{1}{2^{n+1}}$. Let $J_k=I_{m(k+l_n)+j}$ for $k\in\omega$. Then, 
$$|J_k|=|I_{m(k+l_n)+j}|\leq f_k\left(\frac{1}{2^{n+1}}\right)<f_k\left(\varepsilon\right)$$
and the intervals $J_k$, for $k\in\omega$, cover the set $A_j$. 
\end{proof}

We do not know whether compactness is crucial in the above lemma, even in the case of nanoscopic sets.

\begin{problem}
Can every $m$-nanoscopic set be decomposed into $m$ nanoscopic sets?
\end{problem}

\begin{lemma}
\label{lem2}
There is a compact $2$-nanoscopic set which is not nanoscopic.
\end{lemma}

\begin{proof}
In the construction we will need the following two technical partitions of $\omega$ into finite sets:
\begin{itemize}
	\item let $T_{-1}=\{0,1\}$ and $T_i=\{2^{i+1},2^{i+1}+1,\ldots,2^{i+2}-1\}$ for each $i\in\omega$;
	\item let $S_0=T_{-1}$ and $S_{i+1}=\bigcup_{j\in S_i}T_j$ for $i>0$.
\end{itemize}

Let $(I_k)_k$ be a sequence of closed intervals satisfying the following conditions:
\begin{itemize}
	\item if there is $i$ such that $k$ and $n$ both are in $S_i$, then $I_k\cap I_n=\emptyset$;
	\item if $k\in T_n$, then $I_k\subset I_n$;
	\item $|I_{2k}|=|I_{2k+1}|=\left(\frac{1}{2^{2^{1}}}\right)^{2^k}=\frac{1}{2^{2^{k+1}}}$ for all $k$;
	\item for all $i$ the distances between each two intervals from $(I_k)_{k\in T_i}$ are the same and biggest possible.
\end{itemize}
Observe that this construction is possible, i.e., each $I_n$ is long enough to place all $I_k$'s for $k\in T_n$.

The required set is defined by $X=\bigcap_n X_n$, where $X_n=\bigcup_{k\in S_n} I_k$. Clearly, $X$ is compact and $2$-nanoscopic. We will show that $X$ is not nanoscopic. 

Firstly, we need to introduce a function $f\colon\omega\to\omega$ given by $f(k)=2^{i+1}$, where $i\in\omega\cup\{-1\}$ is the unique number such that $k\in T_i$. Observe that for any sequence $(k_j)_j$ satisfying $k_{j+1}\in T_{k_j}$ we have $f(k_j)<f(k_{j+1})$, for all $j$, and $\lim_j f(k_j)=\infty$.

We are ready to show that $X$ is not nanoscopic. Let $\varepsilon=\frac{1}{2^2}$ and take any sequence of intervals $(J_k)_k$ such that $|J_{k}|<(\frac{1}{2^{2}})^{2^k}=\frac{1}{2^{2^{k+1}}}$ for all $k$. We will prove that $J_k$'s cannot cover the whole set $X$ by constructing a decreasing sequence of compact sets $(K_n)_n$ such that each $K_n$ is one of the intervals $I_k$, for $k\in S_n$, and if $K_n=I_k$, then $K_n\cap\bigcup_{i<f(k)}J_i=\emptyset$. It follows from the conditions imposed on $K_n$'s and the properties of the function $f$, that 
$$\emptyset\neq\bigcap_n K_n\subset X\setminus\bigcup_i J_i.$$
Therefore, the construction of the sequence $(K_n)_n$ will conclude the entire proof. 

We inductively define $K_n$'s as follows. There is $m\in T_{-1}$ such that $J_0$ is disjoint with $I_m$. Let $K_0=I_m$. Notice that $K_0\cap\bigcup_{i<f(m)}J_i=K_0\setminus J_0=\emptyset$. Suppose now that $K_i$, for $i\leq n$, are already defined and $K_n=I_j$, for some $j\in S_n$, is such that $K_n\cap\bigcup_{i<f(j)}J_i=\emptyset$. We have
$$X_{n+1}\cap I_j=\bigcup_{k\in T_j} I_k.$$
Hence, the set $X_{n+1}\cap I_j$ is a union of $2^{j+1}$ many intervals $I_{2^{j+1}},\ldots,I_{2^{j+2}-1}$ such that
\begin{enumerate}
	\item they are pairwise disjoint,
	\item each of them is disjoint with $\bigcup_{i<f(j)}J_i$,
	\item each of them has length at least $\frac{1}{2^{2^{2^{j+1}}}}$. 
\end{enumerate}
Therefore, there are only $2^{j+1}-f(j)$ intervals $J_{f(j)},\ldots,J_{2^{j+1}-1}$ such that each of them can completely cover one of the intervals $I_k$, for $k\in T_j$, since $|J_{k}|<(\frac{1}{2^{2}})^{2^k}=\frac{1}{2^{2^{k+1}}}$ for all $k$. Observe also that $0<f(j)\leq 2^{j+1}$. Hence, there are $f(j)$ many $m\in T_j\subset S_{n+1}$ such that $I_m\cap\bigcup_{i<f(m)}J_i=\emptyset$ (since $f(m)=2^{j+1}$ for all $m\in T_j$). Take one of those $m$'s and let $K_{n+1}=I_m$.
\end{proof}

Now we can proceed to the proof of Theorem \ref{tw}.

\begin{proof}
By Lemma \ref{lem2} there is a compact $2$-nanoscopic set $X$ that is not nanoscopic. By Lemma \ref{lem1} there are nanoscopic sets $A$ and $B$ such that $X=A\cup B$. Then $A$ and $B$ are the required sets.
\end{proof}

Of course our proof works for a wider class of sequences of functions satisfying some technical properties. However, it does not lead to a characterization of sequences $(f_n)_n$ for which $\Mf$ is not an ideal -- it does not work for instance in the following case.

\begin{problem}
Let $f_n(\varepsilon)=\frac{\varepsilon}{2^n}$ for all $\varepsilon\in (0,1)$ and $n\in\omega$. Is the family $\Mf$ an ideal?
\end{problem}

Since $\Nano$ is not an ideal, the following question is natural.

\begin{problem}
How does the ideal/$\sigma$-ideal generated by nanoscopic sets look like? Is it of the form $\Mf$ for some sequence of functions $(f_n)_n$?
\end{problem}

\section{Picoscopic sets}

\begin{definition}
A set $M$ is called \emph{picoscopic} if it is $(f_n)$-microscopic for $f_n(\varepsilon)=\varepsilon^{(n+1)!}$ for all $\varepsilon\in (0,1)$ and $n\in\omega$.
\end{definition}

\begin{theorem}
\label{pico-ideal}
The family of picoscopic sets is not an ideal, i.e., there are two picoscopic sets such that their union is not picoscopic.
\end{theorem}

Below we prove an even stronger Theorem \ref{smz}. This Theorem can be proved similarly to Theorem \ref{tw} -- it follows from Lemma \ref{lem1} and the fact that Lemma \ref{lem2} works in the case of picoscopic sets, i.e., there is a compact $2$-picoscopic set which is not picoscopic.

\begin{theorem}
\label{smz}
There are a picoscopic set $X$ and a point $x\in\mathbb{R}$ such that $X\cup\{x\}$ is not picoscopic.
\end{theorem}

For simplicity, we write $F(B)=\bigcup_{n\in B}\{4^n,\ldots,4^{n+1}-1\}$ for $B\subset\omega$. We will need the following technical lemma. It uses some ideas from the Spacing Algorithm for Microscopic Sets proved in \cite{add}.

\begin{lemma}[Spacing Algorithm for Picoscopic Sets]
Let $B\subset\omega$, $m\in\omega$ and $I$ be any interval of length at least $13\left(\frac{1}{13}\right)^{(\min F(B)+1)!}$. Moreover, let $L$ denote the least number such that $13\frac{1}{13^{(L+1)!}}<|I|$. Then in $I$ one can define intervals $I_k$ for all $k\in F(B)$ in such a way that:
\begin{enumerate}
	\item[(i)] $|I_k|=\frac{1}{13^{(m+1)!(k+1)!}}$ for each $k\in F(B)$;
	\item[(ii)] the distance between $I_n$ and $I_k$ is at least $\frac{1}{13^{(l+1)!}}$, where $l=\min\{n,k\}$;
	\item[(iii)] given $b\in B$, the sum of any sequence of intervals $(J_k)_{L\leq k<4^b}$ satisfying $|J_k|\leq\frac{1}{13^{(k+1)!}}$, for all $L\leq k<4^b$, cannot cover more than one third of the intervals $I_k$ for $k\in F(\{b\})$.
\end{enumerate}
\end{lemma}

\begin{proof}
Observe that without loss of generality we can assume that $m=0$ (having defined an interval $I'_k$ of length $\frac{1}{13^{(k+1)!}}$, it suffices to pick any interval $I_k$ of length $\frac{1}{13^{(m+1)!(k+1)!}}$ contained in it). 

Firstly, construct intervals $K_j^i$ for $i\geq L$ and $j<7\cdot 6^{i-L}$. Let $K_j^{L}$ for $j<7$ be such that:
\begin{itemize}
	\item each of them is of length $\frac{1}{13^{(L+1)!}}$;
	\item the distance between each two of them is at least $\frac{1}{13^{(L+1)!}}$;
	\item each of them is contained in $I$.
\end{itemize}
Suppose now that $K_j^i$ for $i<k$ and $j<7\cdot 6^{i-L}$ are defined. Let $K_j^k$ for $j<7\cdot 6^{k-L}$ be such that:
\begin{itemize}
	\item each of them is of length $\frac{1}{13^{(k+1)!}}$;
	\item the distance between each two of them is at least $\frac{1}{13^{(k+1)!}}$;
	\item $K_j^k$ is contained in $K_{r}^{k-1}$, where $r=j\mod 6^{k-L}$.
\end{itemize}
Put 
$$T=\left\{K_s^k:k\geq L\textrm{ and }6\cdot 6^{k-L}\leq s<7\cdot 6^{k-L}\right\}.$$
Note that for each $K_s^k$ belonging to $T$ there is no $K_j^{k+1}$ contained in it. Let $\{K_0,K_1,\ldots\}$ be an enumeration of $T$ with $|K_i|\geq |K_{i+1}|$.\\
Now we can proceed to the definition of $(I_k)_{k\in F(B)}$. Let $\{a_0,a_1,\ldots\}$ be an increasing enumeration of $F(B)$. For each $i$ let $I_{a_i}$ be any interval of length $\frac{1}{13^{(a_i+1)!}}$ contained in $K_i$ (note that $|K_i|\geq |I_{a_i}|$ for all $i$).\\
It is easy to see that the constructed intervals satisfy conditions (i) and (ii). We will show that they satisfy condition (iii) as well.\\
Pick $b\in B$ and consider any sequence of intervals $(J_k)_{L\leq k<4^b}$ with $|J_k|\leq\frac{1}{13^{(k+1)!}}$. Observe that $J_{L}$ can intersect at most $\frac{|F(\{b\})|}{6}$ of the intervals $(I_n)_{n\in F(\{b\})}$. Similarly, $J_{L}\cup J_{L+1}$ can intersect at most $\frac{|F(\{b\})|}{6}+\frac{|F(\{b\})|}{36}$ of the intervals $(I_n)_{n\in F(\{b\})}$. Generally, the sum of the sequence $(J_k)_{L\leq k<4^b}$ can intersect at most $$\frac{|F(\{b\})|}{6}+\ldots+\frac{|F(\{b\})|}{6^{4^b-L}}<\frac{|F(\{b\})|}{3}$$ of the intervals $(I_n)_{n\in F(\{b\})}$. This finishes the proof.
\end{proof}

Now we can proceed to the proof of Theorem \ref{smz}.

\begin{proof}
Set $\varepsilon=\frac{1}{13}$. For all $n\in\omega$ let $\varepsilon_n=\varepsilon^{(n+1)!}$ and $k_n$ be the least number such that $4^{k_n}>(n+1)!$. Denote $h(n)=4^{k_n}$. Firstly, we will construct auxiliary intervals $I^n_k$ for $n,k\in\omega$ with $|I^n_k|=\varepsilon_n^{(k+1)!}$. At the end we will put $X=\bigcap_{i} X_i$, where $X_i=\bigcup_j I^i_j$.

For all $n\in\omega$ and $k<h(n)$ pick closed intervals $I^n_k\subset [0,+\infty)$ with $|I^n_k|=\varepsilon_n^{(k+1)!}$ in such a way that the distance between each two of them is at least $\varepsilon$. The intervals $I^n_k$ for $n\in\omega$ and $k\geq h(n)$ will be defined inductively. In the $m$-th step of the induction we define $I^m_k$ for all $k\geq h(m)$.

At the first step let 
$$T_0=\left\{(n,k):n>0\textrm{ and }k<h(n)\right\}.$$
Pick any partition $(B(n,k))_{(n,k)\in T_0}$ of the set $\omega\setminus k_0$ into infinite sets, such that the length of $I^n_k$ (which is already defined) is at least $13\left(\frac{1}{13}\right)^{(\min F(B(n,k))+1)!}$. Now, for each $(n,k)\in T_0$ apply the Spacing Algorithm to $m=0$, $I^n_k$ and $B(n,k)$ to get closed intervals $I^0_j$ with $|I^0_j|=\varepsilon_0^{(j+1)!}$ for all $j\geq h(0)$.

At the second step let 
$$T_1=\left\{(0,k):I^0_k\cap\bigcup_{i<h(1)}I_i^1=\emptyset\right\}$$
and set any partition $(B(0,k))_{(0,k)\in T_1}$ of the set $\omega\setminus k_1$ into infinite sets, such that:
\begin{itemize}
	\item the length of $I^0_k$ (which is already defined) is at least $13\left(\frac{1}{13}\right)^{(\min F(B(0,k))+1)!}$ for all $(0,k)\in T_1$,
	\item $B(n,k)\cap B(0,k')$ is infinite for all $(n,k)\in T_0$ and $(0,k')\in T_1$.
\end{itemize}
For every $(0,k)\in T_1$ apply the Spacing Algorithm to $m=1$, $I^0_k$ and $B(0,k)$ to get closed intervals $I^1_j$ with $|I^1_j|=\varepsilon_1^{(j+1)!}$ for all $j\geq h(1)$.

Suppose now that $(I^i_j)_{j\in\omega}$ for all $i<m$ are already defined. Let 
$$T_m=\left\{(m-1,k):I^{m-1}_k\cap\bigcup_{i<h(m)}I_i^m=\emptyset\textrm{ and }k\geq h(m-1)\right\}\cup$$ $$\cup\left\{(m-2,k):I^{m-2}_k\cap\bigcup_{i<h(m-1)}I_i^{m-1}\neq\emptyset\right\}.$$
Pick any partition $(B(n,k))_{(n,k)\in T_m}$ of the set $\omega\setminus k_m$ into infinite sets, such that:
\begin{itemize}
	\item the length of $I^n_k$ (which is already defined) is at least $13\left(\frac{1}{13}\right)^{(\min F(B(n,k))+1)!}$ for all $(n,k)\in T_m$,
	\item $B(n,k)\cap B(n',k')$ is infinite for all $(n,k)\in \bigcup_{j<m} T_j$ and $(n',k')\in T_m$.
\end{itemize}
For each $(n,k)\in T_m$ apply the Spacing Algorithm to $m$, $I^n_k$ and $B(n,k)$ to get closed intervals $I^m_j$ with $|I^m_j|=\varepsilon_m^{(j+1)!}$ for all $j\geq h(m)$.

Define $X=\bigcap_{i\in\omega}X_i$, where $X_i=\bigcup_{j\in\omega}I^i_j$ for all $i$. Clearly, $X$ is picoscopic. Let $x=-1$. We will show that $X\cup\{x\}$ is not picoscopic. In order to do it, we must prove that for every $N\in\omega$ the set $X$ cannot be covered by a sequence of intervals $(J_k)$ with $J_N=\emptyset$ and $|J_k|\leq\varepsilon^{(k+1)!}$ for all $k$.

Let $N$ and $(J_k)$ be as above. We will construct sequences $(i_n)$ and $(j_n)$ such that $(I^{i_n}_{j_n})$ is strictly decreasing and each $I^{i_n}_{j_n}$ is disjoint with $\bigcup_{k\leq j_n}J_k$. Therefore, the intersection of all $I^{i_n}_{j_n}$'s will define a point from $X$ which is not covered by the sequence $(J_k)$.

Observe that at least one of $I^N_j$ for $j<h(N)$, say $I^N_K$, is disjoint with $\bigcup_{k\leq N}J_k$.

Consider the case that for all $r,s\in\omega$ with $I_s^r\subset I^N_K$ and $I_s^r\cap\bigcup_{k\leq s}J_k=\emptyset$ there are $j$ and $i$ such that $I_j^i\subsetneq I^r_s$ and $I_j^i\cap\bigcup_{k\leq j}J_k=\emptyset$ (notice that by the construction we have $j>K$). This condition allows us to define the desired sequences. Therefore, we can assume from now on that there are $R,S\in\omega$ with $I_S^R\subset I^N_K$ and $I_S^R\cap\bigcup_{k\leq S}J_k=\emptyset$ such that for all $j$ and $i$ with $I_j^i\subsetneq I^R_S$ we have $I_j^i\cap\bigcup_{k\leq j}J_k\neq\emptyset$.

We will need the following notation. For $i,j\in\omega$ find the unique $n\in\omega$ and $k<h(n)$ such that $I_j^i\subset I^n_k$ and define
$$g(i,j)=\left\{\begin{array}{ll}
\min(\omega\setminus\{n\})&\textrm{ , if }i=n\\
n+1&\textrm{ , if }i=n-1\\
i+1&\textrm{ , if }i\neq n\textrm{ and }i\neq n-1.\\
\end{array} \right.$$
Observe that $g(i,j)\neq n$. The number $g(i,j)$ should be viewed as follows: if $I^{g(i,j)}_l\subset I^i_j$, then there are no $r,s\in\omega$ with $I^{g(i,j)}_l\subsetneq I^r_s\subsetneq I^i_j$.

Denote $M=g(R,S)$ (notice that $M\neq N$) and consider the intervals $I^{M}_j$ for $j\in F(M(R,S))$. By condition (iii) of the Spacing Algorithm and the properties of $I^R_S$, we have:
$$(\star)\ \forall_{b\in B(R,S)}\textrm{ more than one half of }(J_k)_{k\in F(\{b\})}\textrm{ intersect }I^R_S.$$

Now we want to find some $I^M_L$ disjoint with $\bigcup_{k\leq L}J_k$. There are two possible cases. If $M>N$, then at least one of the intervals $I^M_j$, for $j\leq h(M)$, is disjoint with $\bigcup_{k\leq h(M)}J_k$, since $J_N=\emptyset$. If $M<N$, then at least one of the intervals $I^M_j$, for $j\leq N$, is disjoint with $\bigcup_{k\leq N}J_k$. Indeed, if each $J_k$ with $k\leq N$ intersects only one of those intervals, then we are done, since $J_N=\emptyset$. Otherwise, if some $J_k$ with $k\leq N$ intersects more than one of those intervals, then by condition (ii) of the Spacing Algorithm, it can intersect only the ones with $j>k$. Let $\bar{k}$ be the smallest $k$ with this property. Then, it is impossible to intersect each $I^M_j$ for $j\leq \bar{k}$ with $\bigcup_{k<\bar{k}}J_k$. Hence, there exists $I^M_L$ disjoint with $\bigcup_{k<\bar{k}}J_k$. Therefore, in both cases we are able to find $I^M_L$ disjoint with $\bigcup_{k\leq L}J_k$.

We are ready to construct sequences $(i_n)$ and $(j_n)$. Denote $i_0=g(M,L)$ and note that there is some $b_0\in B(M,L)\cap B(R,S)$. By $(\star)$ and condition (iii) of the Spacing Algorithm, there is $j_0\in F(\{b_0\})$ such that $I^{i_0}_{j_0}$ is disjoint with $\bigcup_{k\leq j_0}J_k$. Suppose now that $i_n$ and $j_n$ for $n\leq m$ are defined. Denote $i_{m+1}=g(i_m,j_m)$ and note that there is some $b_{m+1}\in B(i_m,j_m)\cap B(R,S)$. Hence, by $(\star)$ and condition (iii) of the Spacing Algorithm, there is $j_{m+1}\in F(\{b_{m+1}\})$ such that $I^{i_{m+1}}_{j_{m+1}}$ is disjoint with $\bigcup_{k\leq j_{m+1}}J_k$. This ends the construction and the entire proof.
\end{proof}

Of course our proof works for a wider class of sequences of functions satisfying some technical properties. In particular, it works in the following case.

\begin{corollary}
\label{wniosek}
Set $k\in\omega$ and define $f_n(\varepsilon)=\varepsilon^{(n+1)!}$ for $n<k$ and $f_n(\varepsilon)=\varepsilon^{(n+2)!}$ for $n\geq k$. Then there are an $(f_n)$-microscopic set $X$ and a point $x\in\mathbb{R}$ such that $X\cup\{x\}$ is not $(f_n)$-microscopic. In particular, there is a picoscopic set which is not $(f_n)$-microscopic.
\end{corollary}


\bibliographystyle{amsplain}
\bibliography{nanosets}

\end{document}